\documentclass[12pt]{amsart}
\usepackage{amscd, amssymb, graphics}
\usepackage{amsfonts}
\usepackage{amsmath}
\usepackage{amsxtra}
\usepackage{latexsym}
\usepackage[mathcal]{eucal}
\usepackage{graphics,colortbl}
\usepackage{tikz-cd}  
\input xy
\xyoption{all}
\usepackage{epsfig}
\usepackage[pdftex, bookmarks, colorlinks, breaklinks]{hyperref}

\newtheorem{theorem}{Theorem}[section]

\newtheorem{corollary}[theorem]{Corollary}
\newtheorem{lemma}[theorem]{Lemma}
\newtheorem{proposition}[theorem]{Proposition}
\theoremstyle{definition}

\newtheorem{question}[theorem]{Question}
\newtheorem{definition}[theorem]{Definition}
\numberwithin{equation}{section}
\newtheorem{claim}{Claim}
\newtheorem*{theoremA}{Theorem A}
\newtheorem*{theoremB}{Theorem B}
\theoremstyle{remark}
\newtheorem{remark}[theorem]{Remark}
\newtheorem{example}[theorem]{Example}

\newcommand{\ben}{\begin{enumerate}}
\newcommand{\een}{\end{enumerate}}
\newcommand{\bit}{\begin{itemize}}
\newcommand{\eit}{\end{itemize}}

\def\I{{\mathbb{I}}}

\def\U{{\mathcal{U}}}

\def\QED{\nobreak\quad\ifmmode\text{Q.E.D.}\else{\rm Q.E.D.}\fi}

\def\U{{\mathcal{U}}}

\begin{document}

\title[Shadowing in Dynamical Systems]{Shadowing in Dynamical Systems: Zero-dimensional Extensions and Inverse Limits}

\author[D. Peng]{Dekui Peng}

 \address[D. Peng]
	{\hfill\break Institute of Mathematics,
		\hfill\break Nanjing Normal University, 210023,
	\hfill\break China}
\email{pengdk10@lzu.edu.cn}



\maketitle	
\setcounter{tocdepth}{1}

\begin{abstract}
Shifts of finite type (SFTs) play a central role in the theory of shadowing. Good and Meddaugh showed that SFTs serve as basic building blocks in the structural theory of shadowing systems; in particular, every compact metric dynamical system with shadowing is a factor of the inverse limit of an inverse sequence consisting of SFTs. We first show that, for this factor representation alone, neither shadowing nor metrizability is needed: every compact Hausdorff dynamical system is a factor of the inverse limit of an inverse system consisting of SFTs.

Thus, being a factor of an SFT inverse limit is not the structural feature genuinely forced by shadowing. What shadowing provides is stronger stability: in the metric case, every compact shadowing system is a factor of the inverse limit of an inverse sequence of SFTs with surjective bonding maps. Hence the associated zero-dimensional extension still has shadowing. We also prove that every compact Hausdorff shadowing system is conjugate to an inverse limit of metrizable shadowing systems with factor bonding maps. 
\end{abstract}

\maketitle	
\setcounter{tocdepth}{1}

\section{Introduction}

The shadowing property, also known as the pseudo-orbit tracing property, is one
of the standard stability properties in dynamical systems. Roughly speaking, it
says that approximate orbits can be traced by genuine orbits. The notion goes
back to the study of hyperbolic dynamics and stability; see, for example,
Walters \cite{Walters}, Aoki and Hiraide \cite{AH}, and Pilyugin \cite{Pil}.
The behaviour of shadowing under inverse-limit constructions has also been
studied from early on; see, for instance, Chen and Li \cite{CL}. Several
variants, such as asymptotic shadowing and s-limit shadowing, have also been
investigated; see \cite{GOP,MS}.

Symbolic dynamics provides the basic model case. A classical theorem of Walters
\cite{Walters} says that a shift space over a finite alphabet has shadowing if
and only if it is a shift of finite type. Thus shifts of finite type, or SFTs,
are the natural symbolic objects associated with shadowing. Standard references
for symbolic dynamics include Lind and Marcus \cite{LM} and Kitchens \cite{Kit}.

A substantial development was made by Good and Meddaugh \cite{GM}, who showed
that SFTs are not merely examples of shadowing systems but also serve as
building blocks for general shadowing systems through inverse limits. Their
approach uses a topological formulation of shadowing in terms of finite open
covers, which makes sense on arbitrary compact Hausdorff spaces. One of their
basic preservation results is the following.

\begin{theorem}\cite[Theorem 8]{GM}\label{thm:GM8}
Let an inverse system consist of dynamical systems on compact Hausdorff spaces,
each of which has the shadowing property. If the inverse system satisfies the
Mittag-Leffler condition, then its inverse limit also has the shadowing property.
\end{theorem}

The Mittag-Leffler condition is therefore the stability condition that prevents
pseudo-orbits from disappearing at finer levels of an inverse system. Recall
that a topological space is called \emph{zero-dimensional} if it has a base
consisting of clopen sets, that is, sets which are both closed and open. For
compact Hausdorff spaces, zero-dimensionality is equivalent to total
disconnectedness. In this setting, Good and Meddaugh obtained a complete
characterization.

\begin{theorem}\cite[Theorem 18]{GM}\label{thm:GM18}
Let $X$ be a compact, totally disconnected Hausdorff space, and let
$f:X\to X$ be continuous. Then $f$ has the shadowing property if and only if
$(X,f)$ is topologically conjugate to the inverse limit of a Mittag-Leffler
inverse system of shifts of finite type.
\end{theorem}

This theorem belongs to a broader line of work on shadowing in
zero-dimensional and symbolic settings. Kawaguchi \cite{Kaw} studied
topological stability and shadowing for zero-dimensional dynamical systems.
Darji, Gon\c{c}alves and Sobottka \cite{Darji} extended the inverse-limit
approach to complete ultrametric spaces, using shifts of finite order over
countable alphabets. Their proof of the main representation theorem follows,
in essence, the strategy of \cite[Theorem~18]{GM}: one first obtains an
inverse-limit representation by orbit spaces and then uses finite shadowing
to replace these orbit spaces by pseudo-orbit spaces. Lin, Chen and Zhou
\cite{LCZ} considered shadowing for systems of countable group actions, and
Doucha \cite{Dou} recently related shadowing, symbolic dynamics and the strong
topological Rokhlin property for countable group actions.

For compact metric spaces which are not necessarily zero-dimensional, Good and
Meddaugh proved the following representation theorem.

\begin{theorem}\cite[Theorem 20]{GM}\label{thm:GM20}
Let $X$ be a compact metric space and let $f:X\to X$ be a continuous map with
the shadowing property. Then $(X,f)$ is a factor of the inverse limit of an
inverse sequence of shifts of finite type.
\end{theorem}

Theorem \ref{thm:GM20} is an important bridge from zero-dimensional dynamics to
general compact metric dynamics. However, the inverse sequence obtained there is
not known to satisfy the Mittag-Leffler condition. Indeed, the bonding maps
arising from cover refinements are not necessarily surjective. Thus
Theorem \ref{thm:GM20} does not by itself provide a zero-dimensional extension
which still has shadowing.

One purpose of the present paper is to separate the soft part of this phenomenon
from the genuinely shadowing-theoretic part. In Section \ref{revisit}, we revisit the orbit
and pseudo-orbit spaces of Good and Meddaugh in the zero-dimensional setting,
using finite clopen partitions instead of arbitrary finite open covers. We show
that every compact zero-dimensional Hausdorff system is conjugate to the inverse
limit of an inverse system of SFTs; when the original system has shadowing, this
inverse system may be chosen to satisfy the Mittag-Leffler condition. Combining
this with the elementary fact that every compact Hausdorff dynamical system is a
factor of a compact zero-dimensional one, we obtain (see Corollary \ref{coro:fac}):
\[
\text{\emph{every compact Hausdorff dynamical system is a factor of an inverse limit
of SFTs.}}
\]

Thus the existence of an SFT inverse-limit factor representation is not, by
itself, a feature of shadowing. The essential issue is whether the inverse system
can be chosen with an additional stability condition, such as the
Mittag-Leffler condition, or even with surjective bonding maps.

This is the main point of Section \ref{s2}. For compact metric systems with
shadowing, we construct an inverse sequence of SFTs using finite $\eta$-nets
rather than open-cover substitutions. The shadowing property is used to lift
pseudo-orbits at a coarse level to genuine orbits in the original system, and
these genuine orbits can then be approximated at all finer levels. This yields
surjective bonding maps.

\begin{theoremA}
\emph{Let $(X,d)$ be a compact metric space and let $f:X\to X$ be a continuous map
with the shadowing property. Then there exists an inverse sequence
$\{Y_N,p_N^{N+1}\}$ of shifts of finite type with surjective bonding maps such
that $(X,f)$ is a factor of its inverse limit.}
\end{theoremA}

This is Theorem \ref{Zero-ext} below. Since surjective bonding maps imply the
Mittag-Leffler condition, Theorem \ref{thm:GM8} immediately gives a
zero-dimensional compact metric extension which itself has shadowing. In
particular, our result strengthens Theorem \ref{thm:GM20}: the SFT inverse
sequence can be chosen so that no pseudo-orbits are lost at finer symbolic
levels. It also completes the metric characterization in \cite[Corollary 28]{GM}
in terms of ALP factors of Mittag-Leffler inverse sequences of SFTs; see
Corollary \ref{metc}.

In Section \ref{s4}, we turn to arbitrary compact Hausdorff spaces.
Using the unique uniformity of a compact Hausdorff space, we construct metrizable
shadowing factors with prescribed uniform control. This leads to the following
non-metric approximation theorem.

\begin{theoremB}
\emph{Let $(X,f)$ be a compact dynamical system. If $f$ has the shadowing property,
then $(X,f)$ is conjugate to the inverse limit of an inverse system of
metrizable dynamical systems with the shadowing property and factor bonding
maps.}
\end{theoremB}

This is Theorem \ref{ThC}. Thus, although the metric construction with
surjective SFT bonding maps does not directly extend to arbitrary compact
Hausdorff spaces, every compact shadowing system can still be reconstructed from
metrizable shadowing factors.

The preceding results give the following conceptual picture. Good and
Meddaugh showed that the shadowing property is preserved under
Mittag-Leffler inverse limits and under ALP factor
maps (see \cite[Section 5]{GM}). Starting from shifts of finite type, Theorem A shows that
every compact metric shadowing system is obtained by first taking a
Mittag-Leffler inverse limit of SFTs and then passing to an ALP factor. The
non-metrizable case is obtained by one further inverse-limit step: by Theorem
B, every compact Hausdorff shadowing system is the inverse limit of
metrizable shadowing systems with factor bonding maps, and hence of systems
obtained in the previous way.

Thus, if SFTs are regarded as objects of level $0$, then compact
zero-dimensional shadowing systems lie at level at most $1$, compact metric
shadowing systems lie at level at most $2$, and arbitrary compact Hausdorff
shadowing systems lie at level at most $3$, where the allowed operations are
Mittag-Leffler inverse limits and ALP factor maps. This should be understood as
a structural upper bound rather than as a claim that the levels are strict.

Finally, the direct construction of Section~\ref{revisit} also applies beyond
the compact setting. In the last subsection of Section~\ref{final}, we
revisit Section~4.2 of Darji, Gon\c{c}alves and Sobottka \cite{Darji}.
Their representation theorem assumes finite shadowing in order to pass from
orbit spaces to pseudo-orbit spaces. We show instead that every continuous
self-map of a zero-dimensional second-countable space admitting a complete
defining sequence is directly conjugate to the inverse limit of its canonical
one-step pseudo-orbit shifts. If the defining sequence is tame and the map is
uniformly continuous, this conjugacy is uniform. Consequently, the finite
shadowing assumption in \cite[Theorem~4.2.8]{Darji} is unnecessary.
Moreover, for uniformly continuous maps on spaces with complete tame defining
sequences, shadowing is equivalent to the Mittag-Leffler condition for this
canonical inverse sequence. This removes the finite shadowing assumption from
the converse direction of \cite[Theorem~4.2.9(ii)]{Darji}.

Unless otherwise stated, all spaces are assumed to be Hausdorff, and all maps are
continuous.
We denote by $\omega$ the set of natural numbers including $0$; it will also be
identified with the first infinite cardinal. 

By a \emph{dynamical system} we
mean a pair $(X,f)$, where $X$ is a compact Hausdorff space and
$f:X\to X$ is a continuous map.
A \emph{morphism} from a dynamical system $(X,f)$ to a dynamical system
$(Y,g)$ is a continuous map $\pi:X\to Y$ such that $\pi\circ f=g\circ \pi$.
If, in addition, $\pi$ is onto, then $\pi$ is called a \emph{factor map}, and
$(Y,g)$ is called a \emph{factor} of $(X,f)$. If $\pi$ is a homeomorphism, then
we say that $(X,f)$ and $(Y,g)$ are \emph{conjugate}.

\subsection{Shadowing in Metric and Compact Spaces}
 Let us first recall the classical definition of the shadowing property on metric spaces.

\begin{definition}\label{shadmet}
Let $(X, d)$ be a compact metric space and let $f:X \to X$ be a continuous map. For $\delta > 0$, a sequence $\langle x_n : n \in \omega \rangle$ in $X$ is called a \emph{$\delta$-pseudo-orbit} if 
\[
d(f(x_n), x_{n+1}) < \delta
\]
for all $n \in \omega$. Given $\varepsilon > 0$, a point $x \in X$ is said to \emph{$\varepsilon$-shadow} the pseudo-orbit $\langle x_n \rangle$ if 
\[
d(f^n(x), x_n) < \varepsilon
\]
for all $n \in \omega$. The map $f$ has the \emph{shadowing property} if for every $\varepsilon > 0$, there exists a $\delta > 0$ such that every $\delta$-pseudo-orbit is $\varepsilon$-shadowed by some point in $X$.
\end{definition}

To extend this dynamical property to general non-metric topological spaces, Good and Meddaugh \cite{GM} generalized the metric tolerances to the open sets within finite coverings. This formulation allows shadowing to be defined purely in topological terms:

\begin{definition}\label{shadcom}
Let $X$ be a compact space, $f: X \to X$ a continuous map, and $\mathcal{C}$ a finite open cover of $X$. Fix a sequence $(x_n)_{n \in \omega}$ in $X$.
\begin{itemize}
\item The sequence $(x_n)$ is called a \emph{$\mathcal{C}$-pseudo-orbit} if for each $n \in \omega$, there exists some $U_n \in \mathcal{C}$ such that 
\[
\{x_{n+1}, f(x_n)\} \subseteq U_n.
\]
\item A point $x \in X$ is said to \emph{$\mathcal{C}$-shadow} the sequence $\langle x_n \rangle$ if for each $n \in \omega$, there exists some $V_n \in \mathcal{C}$ such that 
\[
\{x_n, f^n(x)\} \subseteq V_n.
\]
\end{itemize}
\end{definition}

\begin{definition}[Topological Shadowing Property]
For a compact space $X$, a continuous map $f: X \to X$ is said to have the \emph{shadowing property} if for every finite open cover $\mathcal{C}$ of $X$, there exists a finite open cover $\mathcal{C}'$ of $X$ such that every $\mathcal{C}'$-pseudo-orbit is $\mathcal{C}$-shadowed by some point in $X$. In this case, we say that $\mathcal{C}'$ \emph{witnesses the shadowing of $\mathcal{C}$}.
\end{definition}

We also recall the basic symbolic terminology used throughout the paper. Let
$A$ be a finite discrete space. The space $A^\omega$, equipped with the product
topology, is compact and zero-dimensional. The left shift
$\sigma:A^\omega\to A^\omega$ is defined by
\[
    \sigma((a_i)_{i\in\omega})=(a_{i+1})_{i\in\omega}.
\]
A \emph{shift space} over $A$ is a closed $\sigma$-invariant subspace of
$A^\omega$. A shift space $Y\subseteq A^\omega$ is called a \emph{shift of
finite type}, or SFT, if there is a finite set $\mathcal F$ of finite words over
$A$ such that $Y$ consists exactly of those sequences in $A^\omega$ in which no
word from $\mathcal F$ occurs.

Equivalently, after passing to a higher block presentation if necessary, an SFT
may be described by finitely many allowed adjacent pairs. Thus, if
$R\subseteq A\times A$, then
\[
    Y_R=\{(a_i)_{i\in\omega}\in A^\omega:(a_i,a_{i+1})\in R
    \text{ for every }i\in\omega\}
\]
is a one-step SFT. This is the form of SFTs that will appear most often below.

\subsection{Inverse Systems and the Mittag-Leffler Condition}
The notion of a sequence-based inverse limit can be generalized to an arbitrary inverse system indexed by a directed pre-ordered set. This provides a powerful framework for constructing complex topological spaces and dynamical systems from simpler discrete components. We refer to \cite{MSeg} for categorical theory of inverse systems.

\begin{definition}
Let $(I,\le)$ be a directed pre-ordered set. An \emph{inverse system (or projective system) of dynamical systems} indexed by $I$ consists of a family
\[
\{(X_i,f_i),\varphi_i^j : i\le j \text{ in } I\},
\]
where each $(X_i,f_i)$ is a dynamical system and each $\varphi_i^j:(X_j,f_j)\to (X_i,f_i)$ is a morphism of dynamical systems (i.e., a continuous map satisfying $\varphi_i^j \circ f_j = f_i \circ \varphi_i^j$), such that:
\begin{itemize}
\item $\varphi_i^i=\mathrm{id}_{X_i}$ for every $i\in I$;
\item whenever $i\le j\le k$, we have $\varphi_i^k = \varphi_i^j\circ \varphi_j^k$.
\end{itemize}
The maps $\varphi_i^j$ are referred to as the \emph{bonding maps}. Usually, we write such an inverse system as $\{(X_i,f_i),\varphi_i^j,I\}$, or simply as $\{(X_i,f_i),\varphi_i^j\}$ when the index set $I$ is clear from the context.
\end{definition}

When the directed index set $I$ is the set of natural numbers $\omega$ (endowed with its standard order), the inverse system is specifically referred to as an \emph{inverse sequence}. In this strictly linear setting, the structural analysis often simplifies significantly.

\begin{definition}
Let $\{(X_i,f_i),\varphi_i^j, I\}$ be an inverse system of dynamical systems.
Its \emph{inverse limit}, denoted by
$\varprojlim \{(X_i,f_i),\varphi_i^j\}$, is the dynamical system $(X,f)$
defined as follows. As a topological space,
\[
X=
\Bigl\{
(x_i)\in \prod_{i\in I}X_i :
\varphi_i^j(x_j)=x_i
\text{ whenever } i\le j
\Bigr\},
\]
endowed with the subspace topology inherited from the product space
$\prod_{i\in I}X_i$. The dynamics $f:X\to X$ is defined coordinatewise by
\[
    f((x_i))=(f_i(x_i)).
\]
For each $i\in I$, the map
\[
    \varphi_i:X\to X_i,\qquad \varphi_i((x_j)_{j\in I})=x_i,
\]
is called the \emph{$i$-th projection map}.
\end{definition}

The inverse limit is characterized by the following universal property, which
will be used repeatedly below.

\begin{proposition}[Universal Property]\label{prop:universal}
Let $(X,f)=\varprojlim \{(X_i,f_i),\varphi_i^j\}$ with $\varphi_i: X\to X_i$ the projection. Suppose that $(Y,g)$ is a dynamical system and that for each $i\in I$ there is a morphism
\[
\psi_i:(Y,g)\to (X_i,f_i)
\]
such that $\psi_i = \varphi_i^j\circ \psi_j$ whenever $i\le j$. Then there exists a unique morphism $\psi:(Y,g)\to (X,f)$ such that $\varphi_i\circ \psi=\psi_i$ for every $i\in I$.
\end{proposition}

A critical property governing the topological stability of an inverse system is the Mittag-Leffler condition. This condition guarantees that the family of images under the bonding maps eventually stabilizes, preventing the inverse limit from collapsing pathologically.

\begin{definition}[Mittag-Leffler Condition]
Let $\{(X_i,f_i),\varphi_i^j, I\}$ be an inverse system indexed by a directed set $I$. The system satisfies the \emph{Mittag-Leffler (ML) condition} if for each $i \in I$, the directed family of image subspaces $\{\varphi_i^j(X_j)\}_{j \ge i}$ eventually stabilizes. That is, for each $i \in I$, there exists an index $k \ge i$ such that for all $j \ge k$,
\[
\varphi_i^j(X_j) = \varphi_i^k(X_k).
\]
\end{definition}

\begin{remark}\label{rmk:trivial_ML}
A fundamental observation is that if all bonding maps $\varphi_i^j$ in an inverse system are \textbf{surjective}, the system trivially satisfies the Mittag-Leffler condition, since $\varphi_i^j(X_j) = X_i$ for all $j \ge i$. In particular, for an inverse sequence, if the adjacent bonding maps $\varphi_n^{n+1}$ are surjective, the ML condition is automatically met. As we will demonstrate in the following subsection, establishing strictly surjective bonding maps is the key to preserving the shadowing property through zero-dimensional extensions.
\end{remark}

\section{Revisiting Theorems \ref{thm:GM18} and \ref{thm:GM20}}\label{revisit}

In this section we revisit the symbolic construction of Good and Meddaugh
from a slightly different point of view. By working directly with finite clopen
partitions in the zero-dimensional case, we obtain a streamlined inverse-limit
representation by shifts of finite type. Together with Theorem \ref{thm:GM8},
this gives a much shorter proof of the main conclusions of Sections 4 and 5 of
\cite{GM}. The reformulation also slightly strengthens the picture: it separates
the mere existence of symbolic inverse-limit representations from the genuinely
shadowing-theoretic content, which is encoded by the Mittag-Leffler condition.

We begin with recalling some notions from \cite{GM}. Since we shall only use them in the zero-dimensional setting, we formulate them in the terms of clopen partitions rather than arbitrary finite open covers.

\begin{definition}
Let $X$ be a compact zero-dimensional Hausdorff space, and let $f:X\to X$ be a
continuous map. For a finite clopen partition $\mathcal U$ of $X$, we define:
\begin{itemize}
\item the \emph{$\mathcal U$-orbit space}, denoted by
$\mathcal O(\mathcal U)$, to be the subspace of $\mathcal U^\omega$ consisting
of all sequences $(U_i)_{i\in\omega}$ such that
$\bigcap_{i\in\omega} f^{-i}(U_i)\neq\varnothing$;

\item the \emph{$\mathcal U$-pseudo-orbit space}, denoted by
$\mathcal{PO}(\mathcal U)$, to be the subspace of $\mathcal U^\omega$ consisting
of all sequences $(U_i)_{i\in\omega}$ such that
$f(U_i)\cap U_{i+1}\neq\varnothing$ for every $i\in\omega$.
\end{itemize}
\end{definition}

It is clear that $\mathcal O(\mathcal U)\subseteq\mathcal{PO}(\mathcal U)$ and
that both spaces are subshifts of $\mathcal U^\omega$. Moreover,
$\mathcal{PO}(\mathcal U)$ is a shift of finite type, since membership is
determined by the allowed adjacent pairs. The itinerary map
$p_{\mathcal U}:X\to\mathcal O(\mathcal U)$, defined by
$p_{\mathcal U}(x)=(U_i)_{i\in\omega}$ where $f^i(x)\in U_i$, is a factor map
from $(X,f)$ onto $(\mathcal O(\mathcal U),\sigma_{\mathcal U})$.

Suppose that $\mathcal V$ is a finite clopen partition refining $\mathcal U$.
For each $V\in\mathcal V$, let $U(V)$ be the unique member of $\mathcal U$
containing $V$. This defines a one-block map
\[
    \iota_{\mathcal U}^{\mathcal V}:\mathcal V^\omega\to\mathcal U^\omega,
    \qquad
    (V_i)_{i\in\omega}\mapsto (U(V_i))_{i\in\omega}.
\]
This map commutes with the shift, and one immediately verifies that
$\iota_{\mathcal U}^{\mathcal V}(\mathcal O(\mathcal V))
=\mathcal O(\mathcal U)$ and
$\iota_{\mathcal U}^{\mathcal V}(\mathcal{PO}(\mathcal V))
\subseteq \mathcal{PO}(\mathcal U)$. Thus we have the following lemma.

\begin{lemma}\label{opo}\cite[Lemma 13]{GM}
Let $\mathcal V$ be a finite clopen partition refining $\mathcal U$. Then
\[
    \mathcal O(\mathcal U)
    =
    \iota_{\mathcal U}^{\mathcal V}(\mathcal O(\mathcal V))
    \subseteq
    \iota_{\mathcal U}^{\mathcal V}(\mathcal{PO}(\mathcal V))
    \subseteq
    \mathcal{PO}(\mathcal U).
\]
\end{lemma}

\begin{lemma}\label{sur}
Let $\mathcal V$ be a finite clopen partition refining $\mathcal U$. If
$\mathcal V$ witnesses $\mathcal U$-shadowing, then
\[
    \iota_{\mathcal U}^{\mathcal V}(\mathcal{PO}(\mathcal V))
    =
    \mathcal O(\mathcal U).
\]
\end{lemma}

\begin{proof}
This is precisely the first part of the proof of \cite[Lemma 15]{GM}. We
include the argument for the convenience of the reader.

By Lemma \ref{opo}, it suffices to prove that
\[
    \iota_{\mathcal U}^{\mathcal V}(\mathcal{PO}(\mathcal V))
    \subseteq \mathcal O(\mathcal U).
\]
Let $(V_i)_{i\in\omega}\in\mathcal{PO}(\mathcal V)$. For each $i\in\omega$,
choose
\[
    x_i\in V_i\cap f^{-1}(V_{i+1}),
\]
which is possible since $f(V_i)\cap V_{i+1}\neq\varnothing$. Then
$(x_i)_{i\in\omega}$ is a $\mathcal V$-pseudo-orbit: indeed,
$f(x_i)$ and $x_{i+1}$ both belong to $V_{i+1}$.

Since $\mathcal V$ witnesses $\mathcal U$-shadowing, there exists $x\in X$
which $\mathcal U$-shadows this pseudo-orbit. Thus, for each $i\in\omega$,
the points $f^i(x)$ and $x_i$ lie in a common member of $\mathcal U$. Since
$x_i\in V_i\subseteq U(V_i)$ and $\mathcal U$ is a partition, this member must
be $U(V_i)$. Hence $f^i(x)\in U(V_i)$ for every $i\in\omega$. Therefore
\[
    x\in \bigcap_{i\in\omega} f^{-i}(U(V_i)),
\]
and consequently
\[
    \iota_{\mathcal U}^{\mathcal V}((V_i)_{i\in\omega})
    =
    (U(V_i))_{i\in\omega}
    \in \mathcal O(\mathcal U).
\]
This proves the desired inclusion.
\end{proof}


Before stating the next theorem, let us clarify its relation to
\cite[Theorem 16]{GM}, which leads to the proof of Theorem \ref{thm:GM18}. In the proof of that theorem, Good and Meddaugh work
under the shadowing assumption and compare two inverse systems: one formed by
the orbit spaces $\mathcal O(\mathcal U)$ and the other by the pseudo-orbit
spaces $\mathcal{PO}(\mathcal U)$. Since the inverse limit of the orbit spaces
is naturally conjugate to the original system, they prove conjugacy by showing
that, under shadowing, the two inverse limits coincide.

The point of the following theorem is slightly different. We show that, the inverse limit of the pseudo-orbit spaces
$\mathcal{PO}(\mathcal U)$ alone is already conjugate to the original system,
without assuming shadowing. Thus shadowing is not needed for the symbolic
inverse-limit representation itself. Its role is instead to ensure the
Mittag-Leffler stability of the inverse system. For this reason, our proof does
not need to introduce the inverse system of orbit spaces.

\begin{theorem}\label{zerolift}
Let $X$ be a compact zero-dimensional Hausdorff space, and let $f:X\to X$ be a
continuous map. Then $(X,f)$ is conjugate to the inverse limit of an inverse
system of shifts of finite type. Moreover, if $f$ has shadowing, then the
inverse system can be chosen to satisfy the Mittag-Leffler condition.
\end{theorem}

\begin{proof}
Let $\mathcal D$ be a cofinal subset of the set of all finite clopen
partitions of $X$, ordered by refinement. For $\mathcal U,\mathcal V\in\mathcal D$
with $\mathcal V$ refining $\mathcal U$, let
$\varphi_{\mathcal U}^{\mathcal V}$ be the restriction of
$\iota_{\mathcal U}^{\mathcal V}$ to $\mathcal{PO}(\mathcal V)$. By Lemma
\ref{opo}, this gives a bonding map
$\varphi_{\mathcal U}^{\mathcal V}:\mathcal{PO}(\mathcal V)\to
\mathcal{PO}(\mathcal U)$. These maps are compatible under refinement, and
therefore
\[
    \{(\mathcal{PO}(\mathcal U),\sigma_{\mathcal U}),
    \varphi_{\mathcal U}^{\mathcal V}\}_{\mathcal U\in\mathcal D}
\]
is an inverse system of shifts of finite type. Let $(Y,g)$ be its inverse limit,
and let $\varphi_{\mathcal U}:Y\to\mathcal{PO}(\mathcal U)$ denote the 
projection.

For each $x\in X$ and each $\mathcal U\in\mathcal D$, let
$p_{\mathcal U}(x)=(U_i)_{i\in\omega}$ be the unique $\mathcal U$-itinerary of
$x$, namely $f^i(x)\in U_i$ for all $i\in\omega$. If $\mathcal V$ refines
$\mathcal U$, then
$\varphi_{\mathcal U}^{\mathcal V}(p_{\mathcal V}(x))=p_{\mathcal U}(x)$.
Since the maps $p_{\mathcal U}$ are evidently continuous, and since each $p_{\mathcal U}$ intertwines $f$ with
$\sigma_{\mathcal U}$, by Proposition \ref{prop:universal}, they induce a continuous map
$\theta:X\to Y$ intertwining $f$ with $g$ such that $p_{\mathcal U} = \varphi_{\mathcal U}\circ \theta$ for any $\mathcal U\in \mathcal D$.

We prove that $\theta$ is a homeomorphism. First, $\theta$ is injective. Indeed,
if $x\neq x'$, then, since $X$ is zero-dimensional Hausdorff and $\mathcal D$ is
cofinal, there exists $\mathcal U\in\mathcal D$ such that $x$ and $x'$ lie in
different members of $\mathcal U$. Hence $p_{\mathcal U}(x)\neq
p_{\mathcal U}(x')$, and therefore $\theta(x)\neq\theta(x')$, because $p_{\mathcal U} = \varphi_{\mathcal U}\circ \theta$.

To prove surjectivity, fix $y\in Y$. Write
$\varphi_{\mathcal U}(y)=(U_i^{\mathcal U})_{i\in\omega}$ for each
$\mathcal U\in\mathcal D$. If $\mathcal V$ refines $\mathcal U$, then
$U_i^{\mathcal V}\subseteq U_i^{\mathcal U}$ for every $i$. Thus, for each
fixed $i$, the family $\{U_i^{\mathcal U}:\mathcal U\in\mathcal D\}$ is
filtered. Since all these sets are nonempty compact clopen sets, their
intersection is nonempty; since finite clopen partitions separate points and
$\mathcal D$ is cofinal, this intersection consists of a single point. Denote it
by $x_i$.

We claim that $f(x_i)=x_{i+1}$ for every $i\in\omega$. Suppose not. Since $X$
is zero-dimensional Hausdorff, choose disjoint clopen neighbourhoods $O$ and
$B$ of $f(x_i)$ and $x_{i+1}$, respectively. By continuity of $f$, there is a
clopen neighbourhood $A$ of $x_i$ such that $f(A)\subseteq O$. Choose
$\mathcal U\in\mathcal D$ refining both $\{A,X\setminus A\}$ and
$\{B,X\setminus B\}$. Then $U_i^{\mathcal U}\subseteq A$ and
$U_{i+1}^{\mathcal U}\subseteq B$, so
$f(U_i^{\mathcal U})\cap U_{i+1}^{\mathcal U}=\varnothing$. This contradicts
$\varphi_{\mathcal U}(y)\in\mathcal{PO}(\mathcal U)$. Hence
$f(x_i)=x_{i+1}$.

It follows that $x_i=f^i(x_0)$ for all $i\in\omega$. Since
$x_i\in U_i^{\mathcal U}$ for every $\mathcal U$ and every $i$, the
$\mathcal U$-itinerary of $x_0$ is exactly $\varphi_{\mathcal U}(y)$. Therefore
$\theta(x_0)=y$. Thus $\theta$ is onto. Since $X$ is compact and $Y$ is
Hausdorff, $\theta$ is a homeomorphism. Hence $(X,f)$ is conjugate to
$(Y,g)$.

Now assume that $f$ has shadowing. For each $\mathcal U\in\mathcal D$, choose
$\mathcal V\in\mathcal D$ refining $\mathcal U$ and witnessing
$\mathcal U$-shadowing.
If $\mathcal W\in\mathcal D$ refines $\mathcal V$, then $\mathcal W$ witnesses $\mathcal U$-shadowing as well.
Then Lemma \ref{sur} gives that
$\varphi_{\mathcal U}^{\mathcal W}(\mathcal{PO}(\mathcal W))
=
\mathcal{O}(\U)$ for any $\mathcal W$ refining $\mathcal V$. Therefore the inverse system
satisfies the Mittag-Leffler condition.
\end{proof}

It is well known that every compact dynamical system is a factor of a
zero-dimensional one. We first recall the underlying topological fact. Let
$X$ be a compact Hausdorff space and put $\kappa=w(X)$, where $w(X)$ denotes the weight of $X$. Then $X$ embeds as a
closed subspace of $\mathbb I^\kappa$, where $\mathbb I=[0,1]$. Since
$\mathbb I$ is a continuous image of the Cantor set $2^\omega$, the cube
$\mathbb I^\kappa$ is a continuous image of $(2^\omega)^\kappa$. If
$q:(2^\omega)^\kappa\to\mathbb I^\kappa$ is such a surjection, then
$
Z=q^{-1}(X)
$
is a closed, and hence compact, zero-dimensional subspace of
$(2^\omega)^\kappa$, and the restriction $q|_Z:Z\to X$ is onto. Moreover,
$w(Z)\leq\kappa$, while the existence of a continuous surjection from $Z$ onto
$X$ implies $w(X)\leq w(Z)$. Thus $w(Z)=w(X)$.

Let $\pi = q|_Z$. Given a continuous map
$f:X\to X$, consider the orbit map
\[
    e:X\to X^\omega,\qquad e(x)=(x,f(x),f^2(x),\ldots).
\]
Its image
\[
    \Delta=\{(x,f(x),f^2(x),\ldots):x\in X\}\subseteq X^\omega
\]
is a compact, and hence closed, shift-invariant subspace of $X^\omega$. Moreover,
$(X,f)$ is conjugate to $(\Delta,\sigma|_\Delta)$ via the map $e$.

Now the map $\pi^\omega:Z^\omega\to X^\omega$ commutes with the shift. Hence
\[
    Y=(\pi^\omega)^{-1}(\Delta)
\]
is a compact zero-dimensional shift-invariant subspace of $Z^\omega$, and the
restriction
\[
    \pi^\omega|_Y:Y\to\Delta
\]
is a factor map. Therefore $(X,f)$ is conjugate to a factor of the compact
zero-dimensional system $(Y,\sigma|_Y)$.

\begin{remark}
If the zero-dimensional extension is not required to have the same weight with $X$, it may alternatively be obtained from the Gleason cover \cite{Gleason,Johnstone}, that is, an irreducible continuous surjection
\[
    p:EX\to X
\]
from an extremally disconnected compact Hausdorff space $EX$ onto $X$. Since extremally disconnected compact Hausdorff spaces are
projective in the category of compact Hausdorff spaces, for every continuous map
$f:X\to X$ the map $f\circ p:EX\to X$ lifts through $p$. Hence there exists a
continuous map $\widetilde f:EX\to EX$ such that
\[
    p\circ \widetilde f=f\circ p.
\]
Thus $(X,f)$ is a factor of the zero-dimensional compact dynamical system
$(EX,\widetilde f)$. This gives a more categorical version of the elementary
construction above. Notice, however, that the lift $\widetilde f$ need not be
unique.
\end{remark}

Theorem \ref{thm:GM20} shows that every compact metric system with shadowing is
a factor of the inverse limit of an inverse sequence of shifts of finite type.
The preceding theorem indicates that the factor-representation part of this
statement is not specific to shadowing in the zero-dimensional case. Combining
this observation with the elementary fact that every compact Hausdorff
dynamical system is a factor of a compact zero-dimensional one, we obtain the
following general consequence. Thus neither shadowing nor metrizability is
needed for the mere existence of an SFT inverse-limit factor representation.

\begin{corollary}\label{coro:fac}
Let $X$ be a compact Hausdorff space and let $f:X\to X$ be continuous. Then
$(X,f)$ is conjugate to a factor of the inverse limit of an inverse system of
shifts of finite type.
\end{corollary}

It is worth noting that if the space $X$ in Theorem \ref{zerolift} is metrizable, then the directed set
$\mathcal D$ can be chosen to be a sequence.
Consequently, the inverse system in Theorem \ref{zerolift}, and hence also the
inverse system appearing in Corollary \ref{coro:fac}, may be chosen to be an
inverse sequence. The latter is because of that $Z$ has the same weight with $X$; so if $X$ is metrizable, then also is $Z$.

We conclude this section with a simple observation concerning sofic shifts.
Recall that a shift space is \emph{sofic} if it is a factor of a shift of
finite type. Since every shift of finite type is a zero-dimensional compact
dynamical system with shadowing, every sofic shift is clearly a factor of such
a system. The converse follows from Lemma \ref{sur}, and yields the following
alternative characterization of sofic shifts. A related result for actions of
countable groups on compact ultrametric spaces was obtained in \cite{Dou}.

\begin{corollary}
A shift space is sofic if and only if it is a factor of a zero-dimensional
compact dynamical system with shadowing.
\end{corollary}

\begin{proof}
Let $(X,f)$ be a zero-dimensional compact dynamical system with shadowing, and
let $\phi:(X,f)\to (Y,\sigma)$ be a factor map, where $Y\subseteq A^\omega$ is a
shift-invariant compact subspace over a finite alphabet $A$. Replacing $A$ by
the set of symbols which actually occur in $Y$, we may assume that the zeroth
coordinate projection $p_0:Y\to A$ is onto.

Put $p_A=p_0\circ\phi:X\to A$. Since $A$ is finite and discrete, the family
\[
    \mathcal U=\{p_A^{-1}(a):a\in A\}
\]
is a finite clopen partition of $X$. The map
\[
    p_A^{-1}(a)\longmapsto a
\]
induces a conjugacy from the $\mathcal U$-orbit space $\mathcal O(\mathcal U)$
onto $Y$. Indeed, for $x\in X$ and $i\in\omega$, the symbol in the $i$-th
coordinate of $\phi(x)$ is precisely $p_A(f^i(x))$.

Since $f$ has shadowing, there exists a finite clopen partition $\mathcal V$
refining $\mathcal U$ and witnessing $\mathcal U$-shadowing. By Lemma
\ref{sur},
\[
    \iota_{\mathcal U}^{\mathcal V}(\mathcal{PO}(\mathcal V))
    =
    \mathcal O(\mathcal U).
\]
The space $\mathcal{PO}(\mathcal V)$ is a shift of finite type, and
$\mathcal O(\mathcal U)$ is its factor. Hence $Y$ is a factor of a shift of
finite type, that is, $Y$ is sofic.
\end{proof}

\section{Metric Shadowing Systems and Surjective SFT Extensions}\label{s2}

In the previous section we separated the soft symbolic representation from the
shadowing phenomenon itself. We showed that every compact Hausdorff dynamical
system is a factor of an inverse limit of shifts of finite type. Thus the mere
existence of an SFT inverse-limit factor representation does not characterize
shadowing, nor does it require metrizability.

The essential question is whether shadowing imposes additional stability on such
a representation. In the zero-dimensional case, Good and Meddaugh proved that
this stability is precisely the Mittag-Leffler condition. For compact metric
systems which are not necessarily zero-dimensional, their representation theorem
produces an SFT inverse-limit extension, but the bonding maps obtained from
cover refinements need not be surjective, and so the Mittag-Leffler condition is
not automatic.

The aim of this section is to show that, in the metric shadowing case, the SFT
inverse sequence can be chosen with surjective bonding maps. The proof uses
finite nets rather than finite open covers. Shadowing allows us to replace
coarse pseudo-orbits by genuine orbits in $X$, and these genuine orbits can then
be approximated coherently at every finer level. 

Throughout this section, an \emph{$\eta$-net} in a metric space $(X,d)$ means a
 subset $A\subseteq X$ such that for every $x\in X$ there exists
$a\in A$ with $d(x,a)<\eta$. We shall write $\eta_n\searrow 0$
to mean that $(\eta_n)_{n\in\omega}$ is a decreasing sequence of positive real
numbers converging to $0$.

\begin{theorem}\label{Zero-ext}
Let $(X,d)$ be a compact metric space and let $f:X\to X$ be a continuous map with the shadowing property. Then there exist an inverse sequence $\{Y_N, p^{N+1}_N\}$ of shifts of finite type with surjective bonding maps, such that $(X, f)$ is a factor of the inverse limit.
\end{theorem}

\begin{proof}
Without loss of generality, we may assume that the metric $d$ is bounded by $1$ (if not, we can replace $d$ with the topologically equivalent metric $\min\{d, 1\}$, which preserves both the compact topology and the shadowing property). Fix the sequence $\varepsilon_n = 1/3^n$ for $n\in\omega$. Note that $\varepsilon_n \searrow 0$ and the sum converges.

Since $f$ has the shadowing property, for each $n$ we may choose $\delta_n>0$ such that every $\delta_n$-pseudo-orbit of $f$ is $\varepsilon_n$-shadowed by a point in $X$. By decreasing $\delta_n$ if necessary, we may assume without loss of generality that $\delta_n \searrow 0$.

For each $n$, choose $\eta_n>0$ satisfying
\begin{equation}\label{eq:eta_bound}
        0 < \eta_n < \min\{\varepsilon_n, \delta_n/3\},
\end{equation}
and such that uniform continuity holds at this scale:
\begin{equation}\label{eq:unif_cont}
        d(x,y) < \eta_n
        \quad\Longrightarrow\quad
        d(f(x),f(y)) < \delta_n/3.
\end{equation}
Let $A_n \subseteq X$ be a finite $\eta_n$-net. Let
\[
        \Sigma_n
        =
        \left\{ (a_i)_{i\geq 0} \in A_n^{\omega} : d(f(a_i),a_{i+1}) < \delta_n \text{ for every } i \geq 0 \right\}.
\]
Then $\Sigma_n$ is a one-sided shift of finite type over the finite alphabet $A_n$. 

For $n\in\omega$, define a vertical compatibility relation $R_{n+1,n} \subseteq A_{n+1} \times A_n$ by
\[
        b \mathrel{R_{n+1,n}} a
        \quad\Longleftrightarrow\quad
        d(b,a) < 2\varepsilon_n.
\]

For each $N\in\omega$, define the subset $Y_N \subseteq \prod_{n=0}^{N} \Sigma_n$ consisting of arrays $y = (a_i^0, a_i^1, \dots, a_i^N)_{i\geq 0} \in Y_N$ if and only if, for each $0 \leq n < N$ and each $i \geq 0$, the vertical relation $a_i^{n+1} \mathrel{R_{n+1,n}} a_i^n$ is satisfied. Consequently, $(Y_N, \sigma_N)$ is a shift of finite type over the finite alphabet $A_0 \times A_1 \times \cdots \times A_N$, where $\sigma_N$ denotes the coordinate-wise shift operator.

Let $p_N^{N+1}: Y_{N+1} \to Y_N$ be the projection that deletes the last coordinate row. We claim that $p_N^{N+1}$ is surjective.

Fix $y = (a_i^0, \dots, a_i^N)_{i\geq 0} \in Y_N$. Since $(a_i^N)_{i\geq 0} \in \Sigma_N$, we have $d(f(a_i^N), a_{i+1}^N) < \delta_N$ for all $i \geq 0$. Thus, $(a_i^N)_{i\geq 0}$ is a $\delta_N$-pseudo-orbit of $f$. By the choice of $\delta_N$, there exists a point $x \in X$ such that
\begin{equation}\label{eq:shadow_x}
        d(f^i(x), a_i^N) < \varepsilon_N
        \qquad\text{for all } i \geq 0.
\end{equation}
Since $A_{N+1}$ is an $\eta_{N+1}$-net, we may choose $a_i^{N+1} \in A_{N+1}$ such that
\begin{equation}\label{eq:eta_net_approx}
        d(a_i^{N+1}, f^i(x)) < \eta_{N+1}
        \qquad\text{for all } i \geq 0.
\end{equation}
Using the triangle inequality along with \eqref{eq:shadow_x} and \eqref{eq:eta_net_approx}, the vertical distance is strictly bounded:
\[
        d(a_i^{N+1}, a_i^N)
        \leq
        d(a_i^{N+1}, f^i(x)) + d(f^i(x), a_i^N)
        <
        \eta_{N+1} + \varepsilon_N
        <
        2\varepsilon_N,
\]
where the last inequality follows from the condition in \eqref{eq:eta_bound}. Hence, $a_i^{N+1} \mathrel{R_{N+1,N}} a_i^N$ for all $i \geq 0$.

Moreover, evaluating the horizontal step using uniform continuity \eqref{eq:unif_cont} together with \eqref{eq:eta_net_approx} for index $i+1$, we obtain:
\[
\begin{aligned}
        d(f(a_i^{N+1}), a_{i+1}^{N+1})
        &\leq
        d(f(a_i^{N+1}), f^{i+1}(x)) + d(f^{i+1}(x), a_{i+1}^{N+1}) \\
        &<
        \delta_{N+1}/3 + \eta_{N+1} \\
        &<
        \delta_{N+1}.
\end{aligned}
\]
Thus, $(a_i^{N+1})_{i\geq 0} \in \Sigma_{N+1}$. Therefore, $(a_i^0, \dots, a_i^N, a_i^{N+1})_{i\geq 0} \in Y_{N+1}$, and it projects exactly to $y$. Hence, $p_N^{N+1}$ is surjective.

Now, define the inverse limit
\[
        Y = \varprojlim (Y_N, p_N^{N+1}).
\]
Since each $Y_N$ is a compact zero-dimensional metric space and the bonding maps are continuous surjections, $Y$ is a compact zero-dimensional metric space. The shift maps $\sigma_N$ commute with the bonding maps, naturally inducing a continuous map $S: Y \to Y$. Consequently, the inverse limit system $(Y,S)$ has shadowing.

We now define the factor map $\pi: Y \to X$. We identify a point of $Y$ with a two-dimensional array $y = (a_i^n)_{n,i\geq 0}$, where $(a_i^n)_{i\geq 0} \in \Sigma_n$ and
\[
        d(a_i^{n+1}, a_i^n) < 2\varepsilon_n
        \qquad\text{for all } n, i \geq 0.
\]
For each fixed $i$, the sequence $(a_i^n)_{n\geq 0}$ is Cauchy in $X$. Indeed, for any $m > n$, substituting $\varepsilon_k = 1/3^k$ explicitly bounds the tail:
\[
        d(a_i^m, a_i^n)
        \leq
        \sum_{k=n}^{m-1} d(a_i^{k+1}, a_i^k)
        <
        2\sum_{k=n}^{m-1} \varepsilon_k
        =
        2\sum_{k=n}^{m-1} \frac{1}{3^k}
        <
        \frac{2 \cdot 3^{-n}}{1 - 1/3}
        =
        3^{1-n}.
\]
Because $X$ is compact, the limit exists; define
\[
        x_i = \lim_{n\to\infty} a_i^n.
\]
For every $n$ and $i$, we know $d(f(a_i^n), a_{i+1}^n) < \delta_n$. Since $\delta_n \searrow 0$, passing to the limit yields $f(x_i) = x_{i+1}$. Thus, $x_i = f^i(x_0)$ for every $i \geq 0$. Define $\pi: Y \to X$ by setting
\[
        \pi(y) = x_0 = \lim_{n\to\infty} a_0^n.
\]
By construction, $\pi(Sy) = x_1 = f(x_0) = f(\pi(y))$, demonstrating that $\pi$ intertwines $S$ and $f$.

To verify continuity, first observe that for every fixed $i$ and every $m>n$,
\[
d(a_i^m,a_i^n)
\le
\sum_{k=n}^{m-1} d(a_i^{k+1},a_i^k)
<
2\sum_{k=n}^{m-1}\varepsilon_k .
\]
Passing to the limit as $m\to\infty$ gives
\[
d(\pi(y),a_i^n(y))
\le
2\sum_{k=n}^{\infty}\varepsilon_k .
\]

Now suppose that $y,z\in Y$ satisfy
\[
a_0^n(y)=a_0^n(z)
\qquad
\text{for all } n\le N .
\]
Then
\[
\begin{aligned}
d(\pi(y),\pi(z))
&\le
d(\pi(y),a_0^N(y))
+d(a_0^N(y),a_0^N(z))
+d(a_0^N(z),\pi(z))\\
&=
d(\pi(y),a_0^N(y))
+d(a_0^N(z),\pi(z))\\
&\le
4\sum_{k=N}^{\infty}\varepsilon_k.
\end{aligned}
\]
Since
$
4\sum_{k=N}^{\infty}\varepsilon_k
=
4\sum_{k=N}^{\infty}\frac1{3^k}
=
2\cdot3^{1-N},
$
the right-hand side tends to $0$ as $N\to\infty$. Therefore $\pi$ is continuous.

Finally, we prove that $\pi$ is surjective. Given any $x \in X$, for each $n$ and each $i \geq 0$, choose $a_i^n \in A_n$ such that $d(a_i^n, f^i(x)) < \eta_n$. Using \eqref{eq:unif_cont}, we have
\[
\begin{aligned}
        d(f(a_i^n), a_{i+1}^n)
        &\leq
        d(f(a_i^n), f^{i+1}(x)) + d(f^{i+1}(x), a_{i+1}^n) \\
        &<
        \delta_n/3 + \eta_n \\
        &<
        \delta_n.
\end{aligned}
\]
Hence, $(a_i^n)_{i\geq 0} \in \Sigma_n$ for every $n$. Furthermore, referencing the hierarchy in \eqref{eq:eta_bound},
\[
        d(a_i^{n+1}, a_i^n)
        \leq
        d(a_i^{n+1}, f^i(x)) + d(f^i(x), a_i^n)
        <
        \eta_{n+1} + \eta_n
        <
        2\varepsilon_n,
\]
so the vertical compatibility condition is strictly satisfied. Therefore, the array $(a_i^n)_{n,i\geq 0}$ defines a valid point $y \in Y$. Taking the limit gives
\[
        \pi(y) = \lim_{n\to\infty} a_0^n = x,
\]
proving $\pi$ is surjective. Therefore, $\pi:(Y,S)\to(X,f)$ is a continuous factor map.
\end{proof}
 
Since the bonding maps in the inverse sequence constructed above are strictly surjective, the system trivially satisfies the Mittag-Leffler condition (as discussed in Remark \ref{rmk:trivial_ML}). Applying Theorem \ref{thm:GM8}, we obtain the following:

\begin{corollary}
Every compact metric dynamical system with the shadowing property is a factor of a zero-dimensional compact metric dynamical system that also possesses the shadowing property.
\end{corollary}

We briefly recall the notion of an ALP factor map introduced by Good and
Meddaugh \cite{GM}. Let $\phi:(Y,g)\to (X,f)$ be a factor map between compact
metric dynamical systems. We say that $\phi$ \emph{almost lifts pseudo-orbits}, or that
$\phi$ is an ALP map, if for every $\varepsilon>0$ and every $\eta>0$ there exists
$\delta>0$ such that every $\delta$-pseudo-orbit $\langle x_i\rangle$ in $X$ admits
an $\eta$-pseudo-orbit $\langle y_i\rangle$ in $Y$ satisfying
$d(\phi(y_i),x_i)<\varepsilon$ for all $i\in\omega$. In other words, pseudo-orbits
in the factor can be lifted, up to a prescribed error, to pseudo-orbits in the
extension.

Good and Meddaugh proved that shadowing implies being an ALP factor of an
inverse sequence of SFTs, and conversely that being an ALP factor of an ML inverse
sequence of SFTs implies shadowing \cite[Corollary 28]{GM}. The missing point was whether the inverse
sequence in the necessary direction could be chosen to satisfy the Mittag-Leffler
condition. Theorem \ref{Zero-ext} gives an affirmative answer in a
stronger form: the bonding maps can be chosen to be surjective. Hence we obtain a
complete metric characterization in terms of ALP factors of ML inverse sequences
of SFTs.

\begin{corollary}\label{metc}
Let $X$ be a compact metric space and let $f: X\to X$ be a continuous map. Then the following assertions are equivalent:
\begin{itemize}
    \item $f$ has the shadowing property;
    \item $(X, f)$ is an ALP factor of a zero-dimensional compact metric dynamical system with the shadowing property;
    \item $(X, f)$ is an ALP factor of the inverse limit of an inverse sequence of shifts of finite type (SFTs) with surjective bonding maps.
\end{itemize}
\end{corollary}


\begin{remark}
It is natural to ask whether the extension obtained in Theorem \ref{Zero-ext}
can be chosen to be trivial when the original space is already
zero-dimensional. More precisely, if $(X,f)$ is a compact zero-dimensional
metric system with shadowing, one may ask whether $(X,f)$ is conjugate to the
inverse limit of an inverse sequence of shifts of finite type whose bonding
maps are surjective.

The construction in Theorem \ref{Zero-ext} does not give such a conclusion.
The point is that the scales used in the proof have two different roles. For
each $n$, the finite set $A_n$ has to be chosen as an $\eta_n$-net, where
$\eta_n$ is subordinate to a pseudo-orbit scale $\delta_n$, while shadowing only
asserts that $\delta_n$-pseudo-orbits are $\varepsilon_n$-shadowed. In general
there is no reason to expect that these scales can be chosen in a
self-compatible way, for instance so that $\varepsilon_n$ itself witnesses
$\varepsilon_n$-shadowing.
\end{remark}

\begin{example}
Let $X=\{0\}\cup\{x_n,y_n:n\geq 1\}$. Define an ultrametric $d$ on $X$ as
follows. Put
\[
    d(0,x_n)=d(0,y_n)=r_n,\qquad d(x_n,y_n)=s_n,
\]
where $r_n=2^{-n}$ and $s_n=2^{-(n+2)}$. If $m\neq n$,
$u\in\{x_n,y_n\}$ and $v\in\{x_m,y_m\}$, set
\[
    d(u,v)=\max\{r_n,r_m\}.
\]
Then $X$ is a compact zero-dimensional metric space, and $x_n\to 0$ and
$y_n\to 0$.

Define $f:X\to X$ by $f(0)=0$, $f(x_n)=y_n$, and $f(y_n)=0$. Then $f$ is
continuous.

We claim that $(X,f)$ has the shadowing property. Let $\varepsilon>0$ be given.
Choose $N\geq 1$ such that $r_N<\varepsilon$, and put
\[
    T_N=\{0\}\cup\{x_n,y_n:n\geq N\}.
\]
Then $T_N$ has diameter at most $r_N$, and hence its diameter is less than
$\varepsilon$. Let
\[
    F_N=X\setminus T_N=\{x_1,y_1,\ldots,x_{N-1},y_{N-1}\}.
\]
Shrinking $\varepsilon$ if necessary,  we may assume
$F_N\neq\varnothing$.

Choose $\delta>0$ such that
\[
    \delta< r_{N-1}
    \quad\text{and}\quad
    \delta<s_j \text{ for every } 1\leq j<N.
\]
We show that every $\delta$-pseudo-orbit is $\varepsilon$-shadowed. Let
$(z_i)_{i\geq 0}$ be a $\delta$-pseudo-orbit, that is,
$d(f(z_i),z_{i+1})<\delta$ for every $i\geq 0$.

First observe that once the pseudo-orbit enters $T_N$, it never leaves $T_N$.
Indeed, $f(T_N)\subseteq T_N$. If $z_i\in T_N$ and $z_{i+1}\in F_N$, then
$d(f(z_i),z_{i+1})\geq r_{N-1}>\delta$, a contradiction.

Now suppose first that $z_0\in T_N$. By the observation above, $z_i\in T_N$ for
all $i\geq 0$. Since $0$ is fixed and $T_N$ has diameter less than
$\varepsilon$, we have $d(f^i(0),z_i)=d(0,z_i)<\varepsilon$ for all $i\geq 0$.
Thus $(z_i)$ is $\varepsilon$-shadowed by $0$.

It remains to consider the case $z_0\in F_N$. If $z_0=x_j$ for some $j<N$, then
$f(z_0)=y_j$. Since $\delta<s_j$ and the nearest point to $y_j$ different from
$y_j$ is at distance at least $s_j$, the condition
$d(f(z_0),z_1)<\delta$ forces $z_1=y_j$. Next, since $f(y_j)=0$ and
$\delta<r_{N-1}$, the condition $d(0,z_2)<\delta$ implies $z_2\in T_N$.
By the previous observation, $z_i\in T_N$ for all $i\geq 2$. Therefore the
genuine orbit of $x_j$ $\varepsilon$-shadows $(z_i)$: indeed,
$f^0(x_j)=z_0$, $f(x_j)=z_1$, and for every $i\geq 2$ we have
$f^i(x_j)=0$ and $z_i\in T_N$, so $d(f^i(x_j),z_i)<\varepsilon$.

Finally, if $z_0=y_j$ for some $j<N$, then $f(z_0)=0$. Since
$\delta<r_{N-1}$, the condition $d(0,z_1)<\delta$ implies $z_1\in T_N$.
Hence $z_i\in T_N$ for all $i\geq 1$. Thus the genuine orbit of $y_j$
$\varepsilon$-shadows $(z_i)$, because $f^0(y_j)=z_0$ and $f^i(y_j)=0$ for
all $i\geq 1$, while $z_i\in T_N$ for all $i\geq 1$.

Therefore every $\delta$-pseudo-orbit is $\varepsilon$-shadowed, and so
$(X,f)$ has shadowing.

However, this system does not admit self-shadowing at arbitrarily small metric
scales for the above ultrametric. More precisely, for every sufficiently small
$\varepsilon>0$, there exists $n$ such that $s_n<\varepsilon\leq r_n$. Then the
constant sequence $(x_n,x_n,x_n,\ldots)$ is an $\varepsilon$-pseudo-orbit, since
\[
    d(f(x_n),x_n)=d(y_n,x_n)=s_n<\varepsilon.
\]
Suppose that this pseudo-orbit were $\varepsilon$-shadowed by some $z\in X$.
Since every orbit in this system eventually reaches $0$, we would have
$f^i(z)=0$ for all sufficiently large $i$. Therefore, for all sufficiently large
$i$,
\[
    d(f^i(z),x_n)=d(0,x_n)=r_n\geq \varepsilon,
\]
a contradiction.

Thus, although $(X,f)$ has shadowing, there is no decreasing sequence
$\varepsilon_k\searrow 0$ such that every $\varepsilon_k$-pseudo-orbit is
$\varepsilon_k$-shadowed with respect to this ultrametric. Consequently, the
injectivity argument suggested after Theorem \ref{Zero-ext} cannot be obtained
merely by replacing the metric with an ultrametric or by refining the finite
nets.

Let us also note that this example does not provide a negative answer to the question of whether every compact zero-dimensional metric system with shadowing can be represented as the inverse limit of shifts of finite type with surjective bonding maps. Indeed, the system above is itself the inverse limit of a sequence of finite dynamical systems with surjective bonding maps, and every finite dynamical system is conjugate to a (one-sided) shift of finite type.
\end{example}

\section{Uniform Approximation by Metrizable Systems}\label{s4}

Theorem \ref{Zero-ext} gives a zero-dimensional shadowing extension for compact metric systems with shadowing, and the extension is obtained as an inverse limit of shifts of finite type with surjective bonding maps. As noted at the end of the previous section, we do not know whether the same conclusion remains valid for arbitrary compact Hausdorff spaces. The main difficulty is that the proof of Theorem \ref{Zero-ext} depends on metric control: finite $\eta$-nets, uniform continuity estimates at numerical scales, and the ability to approximate genuine orbits coherently at all finer metric resolutions. In this section we prove a different, but still useful, reduction theorem for the non-metrizable case. Instead of trying to construct a zero-dimensional shadowing extension directly, we represent a compact Hausdorff shadowing system as an inverse limit of metrizable shadowing systems. The construction uses the unique uniformity of a compact Hausdorff space. Roughly speaking, sufficiently fine closed equivalence relations compatible with the uniform structure produce metrizable factor systems, and the shadowing property passes to these factors in a controlled way. The original system is then recovered as the inverse limit of all such metrizable factors. The result gives the final step in the structural picture described in the introduction. Starting from shifts of finite type, compact zero-dimensional shadowing systems are obtained by Mittag-Leffler inverse limits; compact metric shadowing systems are obtained from these by ALP factor maps; and arbitrary compact Hausdorff shadowing systems are obtained by one further inverse-limit step through metrizable shadowing factors.

\subsection{Uniform Spaces and Uniform Shadowing}

We first recall the basic definitions of uniform spaces \cite{Eng}. Let $X$ be a set. The diagonal of the Cartesian product $X \times X$ is denoted by $\Delta = \{ (x, x) : x \in X \}$. The composition of two relations $A, B \subseteq X \times X$ is defined as 
\[A + B = \{ (x, z) : \exists y \in X \text{ such that } (x, y) \in A \text{ and } (y, z) \in B \}.\]
 For a relation $A$ containing $\Delta$ and $n \in \omega$, $nA$ is defined inductively by $0A = \Delta$ and $(n+1)A = nA + A$. A set $V \subseteq X \times X$ is called an \emph{entourage of the diagonal} if $\Delta \subseteq V$ and $V = -V$ (where $-V$ is the inverse relation). Let $\mathcal{D}_X$ denote the family of all such entourages. 

A \emph{uniformity} on $X$ is a filter $\mathcal{U} \subseteq \mathcal{D}_X$ satisfying the standard axioms of entourages: it is closed under finite intersections, every $V \in \mathcal{U}$ contains an entourage $U \in \mathcal{U}$ such that $2U \subseteq V$, and $\bigcap \mathcal{U} = \Delta$. A family $\mathcal{B} \subseteq \mathcal{U}$ is a \emph{base} for the uniformity if every member of $\mathcal{U}$ contains a member of $\mathcal{B}$.

It is a fundamental result in general topology \cite[Section 8.3]{Eng} that if $(X, \tau)$ is a compact Hausdorff space, there exists a \emph{unique} uniformity $\mathcal{U}$ compatible with the topology $\tau$. This unique uniformity can be explicitly constructed using the collection of all finite open covers $\mathfrak{F}$. Specifically, for any open cover $\mathcal{C} \in \mathfrak{F}$, one associates the entourage $E_{\mathcal{C}} = \bigcup_{C \in \mathcal{C}} (C \times C)$. The family $\{ E_{\mathcal{C}} : \mathcal{C} \in \mathfrak{F} \}$ forms a base for the unique uniformity $\mathcal{U}$. This demonstrates that on a compact space, the uniform structure is completely dictated by its topological covering properties.

Leveraging this uniform structure, we can naturally reformulate the shadowing property without relying on either metrics or explicitly keeping track of open covers.

\begin{definition}[Uniform Shadowing Property]
Let $(X, \mathcal{U})$ be a uniform space, $f: X \to X$ a continuous map, and $U \in \mathcal{U}$. A sequence $\langle x_n \rangle$ is called a \emph{$U$-pseudo-orbit} if $(x_{n+1}, f(x_n)) \in U$ for all $n \in \omega$. A point $x \in X$ is said to \emph{$U$-shadow} $\langle x_n \rangle$ if $(x_n, f^n(x)) \in U$ for all $n \in \omega$. 

The continuous map $f$ is said to have the \emph{shadowing property} if for every $U \in \mathcal{U}$, there exists $V \in \mathcal{U}$ such that every $V$-pseudo-orbit is $U$-shadowed by some point in $X$.
\end{definition}

Due to the aforementioned equivalence between entourages and finite open covers on compact Hausdorff spaces, the uniform shadowing property precisely coincides with the topological shadowing property formulated in Definition \ref{shadcom}.

\subsection{Inverse Approximation by Metrizable Factors}

We now establish the mechanism to approximate general shadowing systems via metrizable factors. 

\begin{lemma}\label{keylemma}
Let $X$ be a compact Hausdorff space and let $f: X\to X$ be continuous. 
If $f$ has the shadowing property, then for any sequence $(V_n)_{n\in \omega}$ of entourages of $X$, 
there exists a factor map $\varphi:(X,f)\to (Z,g)$ such that 
\begin{itemize}
\item[(1)] $Z$ is a metric space;
\item[(2)] $g$ has the shadowing property;
\item[(3)] the metric $d$ on $Z$ is bounded by $2$ and satisfies
\begin{equation}\label{eq:factor-control}
\varphi^{-1}\bigl(\{(z_1,z_2): d(z_1,z_2)<{2^{-(n+2)}}\}\bigr)
\subseteq  V_n,
\end{equation}
for any $n\in \omega$.
\end{itemize}
\end{lemma}

\begin{proof}
Let $U_0=X\times X$. 
Recursively choose a sequence $(U_n)_{n\in \omega}$ of entourages of $X$, such that 
\begin{itemize}
\item $U_{n+1}\subseteq V_n$;
\item $3U_{n+1}\subseteq {U_n}$;
\item $U_{n+1}$ witnesses $\bigcap_{i=0}^n (f\times f)^{-i}(U_n)$-shadowing.
\end{itemize}

It is well known (see, e.g., \cite[Theorem 8.1.10]{Eng}) that for a such a sequence, there exists a uniformly continuous pseudometric $\varrho$ on $X$ such that for each $n\in \omega$,
\begin{equation}\label{e1}
\{(x,y):\varrho(x,y)<2^{-(n+1)}\}
\subseteq U_n
\subseteq \{(x,y):\varrho(x,y)<2^{-n}\}.
\end{equation}
It is clear that $\varrho$ is bounded by $1$.

Let $Y=(X,\varrho)/\sim$, where $x\sim y$ iff $\varrho(x,y)=0$. The equivalence class of $x\in X$ in $Y$ is denoted by $[x]$.

Let $Z$ be the set of all sequences of the form
\[
( [x],[f(x)],[f^2(x)],\ldots).
\]
Define $g:Z\to Z$ by the shift
\[
([x],[f(x)],[f^2(x)],\ldots)
\mapsto
([f(x)],[f^2(x)],[f^3(x)],\ldots).
\]

The map
\[
\varphi:x\mapsto ( [x],[f(x)],[f^2(x)],\ldots)
\]
is continuous and onto, and satisfies $\varphi\circ f=g\circ \varphi$, 
so $\varphi$ is a factor map.

For $x,y\in X$, define
\[
d(\varphi(x),\varphi(y))
=\sum_{n=0}^\infty \frac{1}{2^n}\varrho(f^n(x),f^n(y)).
\]
Then $d$ is a metric on $Z$ bounded by $2$, and $\varphi$ is continuous.
Moreover, if $d(\varphi(x), \varphi(y))<2^{-(n+2)}$, then $\varrho(x, y)<2^{-(n+2)}$. So by (\ref{e1}) we have 
\[(x, y)\in U_{n+1}\subseteq V_n,\]
 and therefore (\ref{eq:factor-control}) is satisfied.

It remains to show that $(Z,g)$ has the shadowing property. 
Let $\varepsilon>0$ with $\varepsilon<1$. 
Choose $k\in \omega$ and $0<\varepsilon'<\tfrac{1}{2}$ such that
\[
\varepsilon'\sum_{n=0}^k \frac{1}{2^n}+\frac{1}{2^k}<\varepsilon.
\]

Take $m\in \omega$ such that $2^{-m}\le \varepsilon'<2^{-(m-1)}$. 
Shrinking $\varepsilon'$ if necessary, assume $m>k$. 
Let 
$
\delta=2^{-(m+2)}.
$

Let $\langle \varphi(x_0),\varphi(x_1),\ldots\rangle$ be a $\delta$-pseudo-orbit in $Z$. 
Then for all $i\in \omega$,
\[
d(\varphi(f(x_i)),\varphi(x_{i+1}))<\delta,
\]
hence
\[
\sum_{n=0}^\infty \frac{1}{2^n}\varrho(f^{n+1}(x_i),f^n(x_{i+1}))<2^{-(m+2)}.
\]
In particular,
\[
\varrho(f(x_i),x_{i+1})<2^{-(m+2)}.
\]
Thus $(x_0,x_1,\ldots)$ is a $U_{m+1}$-pseudo-orbit by (\ref{e1}).

Since $ U_{m+1}$ witnesses $\bigcap_{i=0}^m (f\times f)^{-i}( U_m)$-shadowing and $k<m$, 
there exists $z\in X$ such that for all $n\le k$ and all $i\in \omega$, 
 $(f^n(x_i),f^{n+i}(z))\in U_m$. 
Hence, by (\ref{e1}), we have
\[
\varrho(f^n(x_i),f^{n+i}(z))<2^{-m}.
\]

Fix any $i\in \omega$. Then
\[
\begin{aligned}
d(\varphi(x_i),g^i(\varphi(z)))
&=\sum_{n=0}^\infty \frac{1}{2^n}\varrho(f^n(x_i),f^{n+i}(z))\\
&\le \sum_{n=0}^k \frac{1}{2^n}\varrho(f^n(x_i),f^{n+i}(z)) + \frac{1}{2^k}\\
&< 2^{-m}\sum_{n=0}^k \frac{1}{2^n} + \frac{1}{2^k}\\
&\leq \varepsilon'\sum_{n=0}^k \frac{1}{2^n} + \frac{1}{2^k}\\
&< \varepsilon.
\end{aligned}
\]

Thus $(\varphi(x_0),\varphi(x_1),\ldots)$ is $\varepsilon$-shadowed by $\varphi(z)$.
\end{proof}

\begin{corollary}\label{Coro:shafac}
Let $X$ be a compact space and $f: X\to X$ be a continuous map with the shadowing property. 
Then for any factor map $\psi:(X,f)\to (Y,h)$ with $Y$ metrizable, there exists a factor map $\varphi:(X,f)\to (Z,g)$ such that
\begin{itemize}
\item[(i)] $Z$ is metrizable;
\item[(ii)] $g$ has the shadowing property;
\item[(iii)] $\psi$ factors through $\varphi$, i.e., there exists a factor map $\nu:(Z,g)\to (Y,h)$ such that
\[
\psi=\nu\circ \varphi.
\]
\end{itemize}
\end{corollary}

\begin{proof}
Let $d_Y$ be a compatible metric on $Y$ bounded by $1$. 
For each $n\in \omega$, define
\[
V_n=\{(x_1,x_2)\in X\times X:
d_Y(\psi(x_1),\psi(x_2))<2^{-n}\}.
\]
Then each $V_n$ is an entourage of $X$, by (uniform) continuity of $\varphi$.

Applying Lemma~\ref{keylemma} to $V_n$, we obtain a factor map
\[
\varphi:(X,f)\to (Z,g)
\]
such that:
\begin{itemize}
\item $Z$ is metrizable;
\item $g$ has the shadowing property;
\item there exists a metric $d$ on $Z$ such that
\begin{equation}\label{eq:factor-control-n}
\varphi^{-1}\bigl(\{(z_1,z_2)\in Z\times Z:
d(z_1,z_2)<2^{-(n+2)}\}\bigr)
\subseteq V_n.
\end{equation}
\end{itemize}

We verify (iii). 
Suppose that $\varphi(x_1)=\varphi(x_2)$. 
Then for every $n\in \omega$,
\[
d(\varphi(x_1),\varphi(x_2))=0<2^{-(n+2)}.
\]
Hence, by \eqref{eq:factor-control-n},
$
(x_1,x_2)\in V_n.
$
Therefore,
$
d_Y(\psi(x_1),\psi(x_2))<2^{-n}
$
for every $n\in \omega$, implying that
$
\psi(x_1)=\psi(x_2).
$

Thus one may define a map $\nu:Z\to Y$ by
$
\nu(z)=\psi(x),
$
where $x\in \varphi^{-1}(z)$ is arbitrary. 
The preceding argument shows that $\nu$ is well defined, and clearly
\[
\psi=\nu\circ \varphi.
\]

To prove continuity of $\nu$, let $n\in \omega$. 
If $d(z_1,z_2)<2^{-(n+2)}$, then by \eqref{eq:factor-control-n},
\[
d_Y(\nu(z_1),\nu(z_2))<2^{-n}.
\]
Hence $\nu$ is uniformly continuous.

Moreover, since
$
\psi=\nu\circ\varphi
$
and both $\psi$ and $\varphi$ are surjective, $\nu$ is surjective.

Finally, since both $\psi$ and $\varphi$ intertwine the dynamics,
\[
\nu\circ g=h\circ \nu.
\]
Therefore $\nu:(Z,g)\to (Y,h)$ is a factor map.
\end{proof}

Now we are ready to prove the last result.
\begin{theorem}\label{ThC}
Let $(X,f)$ be a compact dynamical system. 
If $f$ has the shadowing property, then $(X,f)$ is conjugate to the inverse limit of an inverse system of metrizable dynamical systems with the shadowing property and factor bonding maps.
\end{theorem}

\begin{proof}
Consider the family
\[
\{(Z_i,g_i,\varphi_i):i\in I\},
\]
where each $(Z_i,g_i)$ is a metrizable dynamical system with the shadowing property, and
\[
\varphi_i:(X,f)\to (Z_i,g_i)
\]
is a factor map.\footnote{
We identify $(Z_1,g_1,\varphi_1)$ and $(Z_2,g_2,\varphi_2)$ whenever there exists a conjugacy
$
\nu:(Z_1,g_1)\to (Z_2,g_2)
$
such that
$
\varphi_2=\nu\circ \varphi_1.
$
}

\begin{claim}\label{c0}
The family $\{\varphi_i:i\in I\}$ separates points of $X$, that is, for any distinct $x,y\in X$, there exists $i\in I$ such that
\[
\varphi_i(x)\neq \varphi_i(y).
\]
\end{claim}

\begin{proof}[Proof of Claim \ref{c0}]
Let $\I$ be the unit closed interval.
Choose a continuous map
$
\psi:X\to \mathbb I
$
such that $\psi(x)\neq \psi(y)$. Define
\[
\widetilde{\psi}:X\to \mathbb I^\omega,
\qquad
z\mapsto \langle \psi(f^n(z))\rangle_{n\in\omega},
\]
and let $(Y,h)$ be the image of $\widetilde{\psi}$ endowed with the subspace dynamics induced by the shift on $\mathbb I^\omega$. Then $\widetilde{\psi}:(X,f)\to (Y,h)$ is a factor map.

By Corollary~\ref{Coro:shafac}, there exist $i\in I$ and a factor map
$
\nu:(Z_i,g_i)\to (Y,h)
$
such that
$
\widetilde{\psi}
=
\nu\circ \varphi_i.$

Since
\[
\psi
=
\pi_0\circ \widetilde{\psi}
=
\pi_0\circ \nu\circ \varphi_i,
\]
where $\pi_0:Y\to \mathbb I$ denotes the projection onto the $0$-th coordinate, the inequality $\psi(x)\neq \psi(y)$ implies that
$
\varphi_i(x)\neq \varphi_i(y).
$
\end{proof}
Define a partial order on $I$ by setting $i\le j$ if there exists a (necessarily unique) factor map
$
\varphi_i^j:(Z_j,g_j)\to (Z_i,g_i)
$
such that
$
\varphi_i=\varphi_i^j\circ \varphi_j.
$
It is straightforward to verify that $(I,\le)$ is a partially ordered set.

We claim that $(I,\le)$ is directed. 
Let $i,j\in I$. Define
\[
Y=\{(\varphi_i(x),\varphi_j(x)):x\in X\}\subseteq Z_i\times Z_j,
\]
and let
$
g:Y\to Y
$
be given by
\[
g(\varphi_i(x),\varphi_j(x))
=
(\varphi_i(f(x)),\varphi_j(f(x))).
\]
Then $Y$ is metrizable and $g$ is well defined. Moreover, the map $\psi$
\[
x\mapsto (\varphi_i(x),\varphi_j(x))
\]
is a factor map from $(X,f)$ onto $(Y,g)$.

By Corollary~\ref{Coro:shafac}, there exists $k\in I$ such that $\psi$ factors through $\varphi_k$, i.e., there exists a factor map
$
\nu:(Z_k,g_k)\to (Y,g)
$
such that
$
\psi=\nu\circ \varphi_k.
$
Composing with the coordinate projections $\pi_i$ and $\pi_j$ yields
\[
\varphi_i=\pi_i\circ \nu\circ \varphi_k,
\qquad
\varphi_j=\pi_j\circ \nu\circ \varphi_k,
\]
and therefore
\[
i\le k,
\qquad
j\le k.
\]
Hence $(I,\le)$ is directed.

If $i\le j\le k$, then both $\varphi_i^j\circ \varphi_j^k$ and $\varphi_i^k$ are factor maps from $(Z_k,g_k)$ to $(Z_i,g_i)$ satisfying
\[
\varphi_i
=
\varphi_i^j\circ \varphi_j
=
(\varphi_i^j\circ \varphi_j^k)\circ \varphi_k.
\]
By uniqueness of factorization, it follows that
$
\varphi_i^k=\varphi_i^j\circ \varphi_j^k.
$
It follows that the family $\{(Z_i,g_i):i\in I\}$ together with the bonding maps $\{\varphi_i^j:i\le j\}$ forms an inverse system.

Let $(Z,g)$ be the inverse limit of this system, with canonical projections
\[
\psi_i:(Z,g)\to (Z_i,g_i).
\]
Since $\varphi_i=\varphi_i^j\circ \varphi_j$ for all $i\le j$, there exists a unique continuous map
\[
\imath:(X,f)\to (Z,g)
\]
by the universal property, such that
\[
\psi_i\circ \imath=\varphi_i
\quad \text{for all } i\in I,
\qquad
g\circ \imath=\imath\circ f.
\]

We claim that $\imath$ is a homeomorphism. Since $X$ is compact and $Z$ is Hausdorff, it suffices to show that $\imath$ is bijective.

\medskip

\noindent\textbf{Injectivity.}
Let $x_1\ne x_2$ in $X$. Then there exists $i\in I$ such that
\[
\varphi_i(x_1)\ne \varphi_i(x_2),
\]
since the family $\{\varphi_i\}$ separates points of $X$, by Claim \ref{c0}. Hence
\[
\psi_i(\imath(x_1))\ne \psi_i(\imath(x_2)),
\]
so $\imath(x_1)\ne \imath(x_2)$.

\medskip

\noindent\textbf{Surjectivity.}
Let $z\in Z$. For each $i\in I$, define
$
X_i=\varphi_i^{-1}(\psi_i(z)).
$
If $i\leq j$, and $x\in X_j$, then 
\[\varphi_i(x)=\varphi^j_i(\varphi_j(x))=\varphi^j_i(\psi_j(z))=\psi_i(z).\]
Hence we have $X_j\subseteq X_i$.

We have also that each $X_i$ is nonempty since $\varphi_i$ is surjective. Moreover, if $J\subseteq I$ is finite, then by directedness there exists $k\in I$ such that $j\le k$ for all $j\in J$, hence
\[
X_k\subseteq \bigcap_{j\in J} X_j.
\]
Thus $\{X_i:i\in I\}$ has the finite intersection property. By compactness of $X$,
\[
\bigcap_{i\in I} X_i \ne \varnothing.
\]
Take $x$ in this intersection. Then $\psi_i(\imath(x))=\psi_i(z)$ for all $i$, hence $\imath(x)=z$.
\end{proof}

\section{Final Comments}\label{final}

\subsection{The non-metrizable extension problem}

We now turn to a possible non-metrizable analogue of
Theorem~\ref{Zero-ext}.
 We do not know whether the conclusion of that theorem
remains true for arbitrary compact Hausdorff spaces. More explicitly, the
following question is open to us.

\begin{question}\label{q:nonmetric_zero_ext}
Let $(X,f)$ be a compact Hausdorff dynamical system with shadowing. Does there
exist a compact zero-dimensional dynamical system $(Y,g)$ with shadowing and a
factor map $\pi:(Y,g)\to (X,f)$?
\end{question}

The metric proof of Theorem \ref{Zero-ext} uses a sequence of finer and finer
scales. In the non-metrizable case the natural set of scales need not contain a
cofinal sequence, and this is the main new obstruction.

Let $(X,\U)$ be a compact uniform space of weight
$\kappa\geq\omega$. We shall use the following standard viewpoint from Tukey
order; see Tukey \cite{Tukey} and
Todor\v{c}evi\'c \cite{Todorcevic}. A directed set of cofinality $\kappa$ is Tukey below
$[\kappa]^{<\omega}$, the set of finite subsets of $\kappa$ ordered by
inclusion. Thus, after passing to a cofinal base of the uniformity, we may index
the relevant entourages by $[\kappa]^{<\omega}$. We write this base as
$\{U_F:F\in[\kappa]^{<\omega}\}$, with $U_\varnothing=X\times X$, and choose it
so that if $F\subseteq G$, then $U_G$ is sufficiently small with respect to
$U_F$.

Assume now that $f:X\to X$ has shadowing. For each
$F\in[\kappa]^{<\omega}$, choose an entourage $V_F\subseteq U_F$ witnessing
$U_F$-shadowing. Choose also entourages $V'_F$ and $W_F$ such that
$3V'_F\subseteq V_F$ and
\[
    (x,y)\in W_F \quad\Longrightarrow\quad (f(x),f(y))\in V'_F.
\]
Since $X$ is compact, fix a finite $W_F$-net $A_F$ in $X$. Define
\[
    \Sigma_F=\{(a_i)_{i\in\omega}\in A_F^\omega:
    (f(a_i),a_{i+1})\in V_F\text{ for every }i\in\omega\}.
\]
Then $\Sigma_F$ is a shift of finite type.

A natural attempt is to imitate Theorem \ref{Zero-ext} by considering, for each
finite $F\subseteq\kappa$, arrays indexed by all subsets of $F$. Namely, one
may define $Y_F$ as a subshift of $\prod_{G\subseteq F}\Sigma_G$ consisting of
arrays $[a_i^G]_{G\subseteq F,\ i\in\omega}$ satisfying suitable vertical
coherence conditions. For instance, one may require
\[
    (a_i^G,a_i^H)\in 2U_G
    \quad\text{whenever }G\subseteq H\subseteq F\text{ and }i\in\omega.
\]
Since $F$ is finite, these are only finitely many zero-step conditions, together
with the one-step conditions defining the rows $\Sigma_G$. Hence each $Y_F$ is
again a shift of finite type.

If $F\subseteq F'$, there is a natural projection
$p_F^{F'}:Y_{F'}\to Y_F$ obtained by deleting the rows indexed by subsets not
contained in $F$. These maps commute with the shifts and form an inverse system.
The key question is whether they can be made surjective.

\begin{question}\label{q:surj_cubes}
Can the above construction be arranged so that, for all
$F\subseteq F'$ in $[\kappa]^{<\omega}$, the projection
$p_F^{F'}:Y_{F'}\to Y_F$ is surjective?
\end{question}

It is enough to understand the case $F'=F\cup\{\alpha\}$. Given an element of
$Y_F$, one would like to construct the new rows indexed by
$G\cup\{\alpha\}$, where $G\subseteq F$. The obvious strategy is to shadow each
old row $(a_i^G)_{i\in\omega}$ by an orbit $(f^i(x_G))_{i\in\omega}$ in $X$,
and then approximate this orbit by points in the finer net
$A_{G\cup\{\alpha\}}$. This gives a row in $\Sigma_{G\cup\{\alpha\}}$ which is
vertically close to the old row indexed by $G$.

The difficulty is that this does not automatically give coherence between the
new rows themselves. Already when $F=\{\beta\}$ and
$F'=\{\alpha,\beta\}$, one has to construct a row indexed by
$\{\alpha,\beta\}$ which is simultaneously compatible with the rows indexed by
$\{\alpha\}$ and $\{\beta\}$. The shadowing points used to construct these two
rows may be unrelated. Thus one is not merely lifting along a chain of scales;
one has to fill finite cubes of compatible symbolic data.

This is precisely where the Tukey type of the non-metrizable uniformity enters.
In the metric case the relevant scales may be taken along a cofinal sequence,
so each step has only one previous level to respect. For a general compact
Hausdorff space of weight $\kappa$, the natural cofinal type is
$[\kappa]^{<\omega}$, and passing from $F$ to $F\cup\{\alpha\}$ introduces a
whole family of new faces. The resulting cube-filling problem is not forced by
the shadowing property in any obvious way.

A positive answer to Question \ref{q:surj_cubes} would imply a positive answer
to Question \ref{q:nonmetric_zero_ext}. Indeed, the inverse limit
\[
    Y=\varprojlim\{(Y_F,\sigma_F),p_F^{F'}\}
\]
would then be an inverse limit of shifts of finite type with surjective bonding
maps, and hence a compact zero-dimensional system with shadowing. The usual
Cauchy-net argument, using the cofinal family $\{U_F\}$, would then define a
factor map from $Y$ onto $X$.

At present we do not know whether such a cube-filling argument can always be
carried out. It is possible that one needs a more refined construction of the
systems $Y_F$, or that Question \ref{q:nonmetric_zero_ext} has a negative
answer in full generality.

\subsection{The non-compact non-Archimedean case}
\label{subsec:nonarch}

In this subsection, we revisit Section~4.2 of Darji, Gon\c{c}alves and
Sobottka \cite{Darji}. We show that the hypotheses of their main
representation results, Theorems~4.2.8 and~4.2.9, can be weakened.
The proof of \cite[Theorem~4.2.8]{Darji} still proceeds, in essence, along
the lines of the argument used in \cite[Theorem~18]{GM}: one first obtains
an inverse-limit representation by orbit spaces and then uses finite
shadowing to replace these orbit spaces by pseudo-orbit spaces. By contrast,
the method developed in Section~\ref{revisit} gives a direct inverse-limit
representation by pseudo-orbit spaces. In the present non-compact setting,
compactness is replaced by completeness of a defining sequence. As a
consequence, finite shadowing is not needed for the representation itself,
and the assumptions in both \cite[Theorems~4.2.8 and~4.2.9]{Darji} can be
reduced.

Their use of tame defining sequences also admits a natural uniform
interpretation: the spaces under consideration are precisely separable
complete metrizable spaces whose uniformity has a countable base consisting
of equivalence relations. In other words, they are complete metrizable
\emph{non-Archimedean} uniform spaces.

 Let
$X$ be a zero-dimensional second-countable space. A \emph{defining sequence}
on $X$ is a sequence $\mathcal A=\{\mathcal U_n:n\in\omega\}$ of countable
clopen partitions such that $\mathcal U_{n+1}$ refines $\mathcal U_n$ and
$\bigcup_{n\in\omega}\mathcal U_n$ is a topological base for $X$. It is called
\emph{complete} if, whenever $U_n\in\mathcal U_n$ and
$U_{n+1}\subseteq U_n$ for every $n$, the intersection
$\bigcap_{n\in\omega}U_n$ is nonempty. Since the partitions form a base,
this intersection then consists of exactly one point.

Each partition $\mathcal U_n$ determines the equivalence relation
$E_n=\bigcup_{U\in\mathcal U_n}U\times U$. The family $\{E_n:n\in\omega\}$
is a base for a compatible non-Archimedean uniformity. Equivalently, it is
generated by the ultrametric $u_{\mathcal A}$ given by
$u_{\mathcal A}(x,x)=0$ and
\[
u_{\mathcal A}(x,y)
=
\frac{1}{1+\min\{n:\mathcal U_n[x]\neq\mathcal U_n[y]\}}
\qquad (x\neq y),
\]
where $\mathcal U_n[x]$ denotes the member of $\mathcal U_n$ containing
$x$. The defining sequence is complete if and only if
$u_{\mathcal A}$ is complete.

Suppose now that $d$ is a compatible metric on $X$. Following
\cite{Darji}, the sequence $\mathcal A$ is called \emph{tame} with respect
to $d$ if
\[
S_n=\sup\{\operatorname{diam}_d(U):U\in\mathcal U_n\}\longrightarrow 0
\]
and, for every $n$, there is $\rho_n>0$ such that
$d(U,V)\geq\rho_n$ whenever $U$ and $V$ are distinct members of
$\mathcal U_n$. This is equivalent to saying that $d$ and
$u_{\mathcal A}$ generate the same uniformity.

Let $f:X\to X$ be continuous. For each $n\in\omega$, define
$\mathcal{PO}(\mathcal U_n)$ to be the subspace of
$\mathcal U_n^\omega$ consisting of all sequences
$(U_i)_{i\in\omega}$ satisfying
$f(U_i)\cap U_{i+1}\neq\varnothing$ for every $i\in\omega$.
Thus $\mathcal{PO}(\mathcal U_n)$ is a one-step shift over the countable
alphabet $\mathcal U_n$. These shifts are called \emph{shifts of finite order} in
\cite{Darji}; in general they need not be shifts of finite type in the
finite-alphabet sense used earlier in this paper.

If $m\geq n$, refinement induces a one-block map
$\varphi_n^m:\mathcal{PO}(\mathcal U_m)\to
\mathcal{PO}(\mathcal U_n)$, obtained by replacing every member of
$\mathcal U_m$ by the unique member of $\mathcal U_n$ containing it. The
maps $\varphi_n^m$ commute with the shifts and form an inverse sequence.

We equip each alphabet $\mathcal U_n$ with the discrete uniformity and
$\mathcal U_n^\omega$ with the corresponding product uniformity. The shift
space $\mathcal{PO}(\mathcal U_n)$ is endowed with the induced subspace
uniformity, and the inverse limit is endowed with the subspace uniformity
inherited from the product $\prod_{n\in\omega}\mathcal{PO}(\mathcal U_n)$.
In particular, its underlying topology is the usual inverse-limit topology.

More explicitly, let $d_n$ be the usual prodiscrete ultrametric on
$\mathcal{PO}(\mathcal U_n)$, given by $d_n(\mathbf U, \mathbf V)=0$ if
$\mathbf U= \mathbf V$, and otherwise by
$d_n(\mathbf U, \mathbf V)=(1+k)^{-1}$, where $k$ is the least coordinate at
which $\mathbf U$ and $\mathbf V$ differ. The product uniformity may then be
metrized by
\[
d_{\Pi}(y,z)=\sup_{n\in\omega}2^{-n}d_n(y_n,z_n),
\]
and the inverse limit carries the restriction of $d_{\Pi}$.

The following is the direct non-compact counterpart of
Theorem~\ref{zerolift}.

\begin{theorem}\label{thm:nonarch-representation}
Let $X$ admit a complete defining sequence
$\mathcal A=\{\mathcal U_n:n\in\omega\}$, and let $f:X\to X$ be continuous.
Then $(X,f)$ is topologically conjugate to
\[
\varprojlim
\{(\mathcal{PO}(\mathcal U_n),\sigma_n),\varphi_n^m:
n\leq m<\omega\}.
\]
If $X$ is equipped with a metric $d$, the sequence $\mathcal A$ is tame
with respect to $d$, and $f$ is uniformly continuous, then the conjugacy
and its inverse are uniformly continuous.
\end{theorem}

\begin{proof}
Let $(Y,g)$ denote the inverse limit in the statement. Define
$\theta:X\to Y$ by
\[
\theta(x)
=
\bigl((\mathcal U_n[f^i(x)])_{i\in\omega}\bigr)_{n\in\omega}.
\]
The family on the right is compatible under refinement, and each of its
coordinates belongs to $\mathcal{PO}(\mathcal U_n)$. Moreover,
$\theta\circ f=g\circ\theta$.

The map $\theta$ is injective because the members of the partitions
$\mathcal U_n$ form a base and hence separate points. Its continuity follows
from the continuity of the finite iterates of $f$: prescribing finitely many
coordinates of finitely many itineraries amounts to requiring that finitely
many iterates of a point belong to prescribed clopen sets.

We prove that $\theta$ is onto. Let
$y=((U_i^n)_{i\in\omega})_{n\in\omega}\in Y$. Compatibility under refinement
gives $U_i^{n+1}\subseteq U_i^n$ for every $i,n\in\omega$. By completeness
of the defining sequence, for each fixed $i$ there is a unique point $x_i$
such that
$x_i\in\bigcap_{n\in\omega}U_i^n$.

We claim that $f(x_i)=x_{i+1}$ for every $i$. Suppose otherwise. Choose
disjoint clopen neighbourhoods $O$ and $B$ of $f(x_i)$ and $x_{i+1}$,
respectively. By continuity of $f$, there is a clopen neighbourhood $A$ of
$x_i$ such that $f(A)\subseteq O$. Since the partitions form a base, for
some sufficiently large $n$ we have $U_i^n\subseteq A$ and
$U_{i+1}^n\subseteq B$. Consequently,
$f(U_i^n)\cap U_{i+1}^n=\varnothing$, contradicting
$(U_j^n)_{j\in\omega}\in\mathcal{PO}(\mathcal U_n)$. Hence
$f(x_i)=x_{i+1}$.

It follows that $x_i=f^i(x_0)$ for every $i$, and therefore
$\theta(x_0)=y$. Thus $\theta$ is bijective.

To see directly that $\theta^{-1}$ is continuous, let $O$ be a neighbourhood
of $x=\theta^{-1}(y)$. Choose $n$ such that
$\mathcal U_n[x]\subseteq O$. Every element of $Y$ whose zeroth symbol at
level $n$ agrees with that of $y$ is mapped by $\theta^{-1}$ into
$\mathcal U_n[x]$. Hence $\theta^{-1}$ is continuous.

Assume now that $\mathcal A$ is tame with respect to $d$ and that $f$ is
uniformly continuous. We show that $\theta$ is uniformly continuous with
respect to $d$ and $d_{\Pi}$.

Let $\varepsilon>0$. Choose $N,K\in\omega$ such that
$2^{-(N+1)}<\varepsilon$ and $(K+1)^{-1}<\varepsilon$. Set
$\rho=\min\{\rho_n:n\leq N\}$. Since $f$ is uniformly continuous, so are
the finitely many maps $f^i$, $0\leq i<K$. Hence there exists $\delta>0$
such that
\[
d(x,x')<\delta
\quad\Longrightarrow\quad
d(f^i(x),f^i(x'))<\rho
\quad\text{for every }0\leq i<K.
\]
If $n\leq N$, distinct members of $\mathcal U_n$ are
$\rho_n$-separated. Since $\rho\leq\rho_n$, it follows that
$\mathcal U_n[f^i(x)]=\mathcal U_n[f^i(x')]$ for every $n\leq N$ and
$0\leq i<K$. Thus the first $K$ symbols of the $n$th coordinates of
$\theta(x)$ and $\theta(x')$ agree, and consequently
$d_n(\theta(x)_n,\theta(x')_n)\leq(K+1)^{-1}$ for every $n\leq N$.

For $n>N$, we have
$2^{-n}d_n(\theta(x)_n,\theta(x')_n)\leq 2^{-n}\leq 2^{-(N+1)}$.
Therefore
\[
d_{\Pi}(\theta(x),\theta(x'))
\leq
\max\{(K+1)^{-1},2^{-(N+1)}\}
<\varepsilon.
\]
This proves that $\theta$ is uniformly continuous.

We finally prove that $\theta^{-1}$ is uniformly continuous. Let
$\varepsilon>0$. Since $\mathcal A$ is tame, choose $n\in\omega$ such
that $S_n<\varepsilon$, and put $\eta=2^{-n}$. Suppose that
$y,z\in Y$ satisfy $d_{\Pi}(y,z)<\eta$. Then
\[
2^{-n}d_n(y_n,z_n)<2^{-n},
\]
and hence $d_n(y_n,z_n)<1$. By the definition of the prodiscrete metric,
the zeroth symbols of $y_n$ and $z_n$ agree.

Let $x=\theta^{-1}(y)$ and $x'=\theta^{-1}(z)$. The equality of these
zeroth symbols means that $x$ and $x'$ belong to the same member of
$\mathcal U_n$. It follows that
$d(x,x')\leq S_n<\varepsilon$. Hence $\theta^{-1}$ is uniformly continuous,
and $\theta$ is a uniform conjugacy.
\end{proof}

The proof is parallel to that of Theorem~\ref{zerolift}. Compactness was used
there to ensure that a compatible family of partition elements has a
nonempty intersection. In the present setting, this is exactly the role
played by completeness of the defining sequence.

Theorem~\ref{thm:nonarch-representation} also separates the symbolic
representation from the shadowing property. In particular, the finite
shadowing hypothesis in \cite[Theorem~4.2.8]{Darji} is not needed for the
inverse-limit representation itself. Shadowing is instead detected by the
Mittag-Leffler condition.

The following result sharpens \cite[Theorem~4.2.9]{Darji}. The forward
implication requires neither completeness of the defining sequence nor
uniform continuity of the map, while in the converse direction our direct
inverse-limit representation allows us to remove the finite shadowing
assumption.

\begin{corollary}\label{cor:nonarch-ML}
Let $(X,d,f)$ be a dynamical system, and let
$\mathcal A=\{\mathcal U_n:n\in\omega\}$ be a tame defining sequence of
$(X,d)$. Consider the canonical inverse sequence
\[
\bigl\{(\mathcal{PO}(\mathcal U_n),\sigma_n),\varphi_n^m:
n\leq m<\omega\bigr\}.
\]
Then the following statements hold.
\begin{enumerate}
\item[(i)] If $f$ has the shadowing property, then the canonical inverse
sequence satisfies the Mittag--Leffler condition.

\item[(ii)] Suppose, in addition, that $\mathcal A$ is complete and that
$f$ is uniformly continuous. If the canonical inverse sequence satisfies
the Mittag--Leffler condition, then $f$ has the shadowing property.
\end{enumerate}

Consequently, if $\mathcal A$ is complete and $f$ is uniformly continuous,
then $f$ has the shadowing property if and only if the canonical inverse
sequence satisfies the Mittag--Leffler condition.
\end{corollary}

\begin{proof}
(i) has been proved in \cite[Theorem 4.2.9(i)]{Darji}.

For (ii), Theorem~\ref{thm:nonarch-representation} gives a uniform
conjugacy between $(X,d,f)$ and the inverse limit of the canonical inverse
sequence. Each $\mathcal{PO}(\mathcal U_n)$ is a one-step shift over a
countable alphabet and therefore has the shadowing property by
\cite[Proposition~2.3.5]{Darji}. The bonding maps are uniformly continuous,
and the inverse sequence satisfies the Mittag--Leffler condition. Hence its
inverse limit has the shadowing property by
\cite[Theorem~4.1.7]{Darji}. Since shadowing is preserved by uniform
conjugacy, $f$ has the shadowing property.
\end{proof}

\begin{remark}\label{rem:general-nonarch}
Although we have formulated the results of this subsection in the metrizable
setting, metrizability is not essential. The same arguments apply to an
arbitrary Hausdorff non-Archimedean uniform space.

More precisely, one replaces the countable tame defining sequence by a
directed base $\mathcal B$ of equivalence entourages. For each
$E\in\mathcal B$, the set of $E$-classes gives a discrete clopen partition
$\mathcal U_E$, and the corresponding pseudo-orbit shifts form an inverse
system indexed by $\mathcal B$, ordered by reverse inclusion. Completeness
of the defining sequence is replaced by completeness of the uniform space,
while metric shadowing and uniform conjugacy are replaced by their
entourage-based versions. With these modifications,
Theorem~\ref{thm:nonarch-representation} and
Corollary~\ref{cor:nonarch-ML}, including the weaker assumptions in
part~{\rm (i)}, remain valid. We omit the routine details.
\end{remark}



\end{document}